\documentclass[11pt,reqno]{amsart} % for arxiv

\usepackage[T1]{fontenc}
\usepackage{tgpagella}
\usepackage{mathpazo}

\usepackage[pagebackref]{hyperref} % for pdflatex
\usepackage{xcolor}
\usepackage{amsmath,amsthm}
\usepackage{amssymb}
\usepackage{graphicx}
\usepackage[mathscr]{eucal}
\usepackage{MnSymbol}
\usepackage[frame,ps,matrix,arrow,curve,rotate,all,2cell,tips]{xy}
\usepackage{epic,eepic}
\usepackage{enumerate}

\newtheorem{main}{Theorem}
\newtheorem{mainCor}[main]{Corollary} % this makes it Corollary B

\newtheorem{theorem}[equation]{Theorem}
\newtheorem{lemma}[equation]{Lemma}

\newtheorem{proposition}[equation]{Proposition}
\newtheorem{corollary}[equation]{Corollary}

\theoremstyle{definition}
\newtheorem{definition}[equation]{Definition}

\newtheorem{remark}[equation]{Remark}

\numberwithin{equation}{section}

\newcommand{\arrowo}{\overrightarrow{\mathcal{O}}}

\newcommand{\M}{\mathsf{M}}

\newcommand{\sD}{\mathscr{D}}

\DeclareMathOperator{\colim}{colim}

\DeclareMathOperator{\Arr}{Arr}

\newcommand{\po}{\ar@{}[dr]|(.7){\Searrow}}
\newcommand{\pb}{\ar@{}[dr]|(.3){\Nwarrow}}
\newcommand{\Ch}{\mathsf{Ch}}

\newcommand{\Chz}{\Ch(\mathbb{Z})}

\newcommand{\cat}[1]{\mathcal{#1}}

\newcommand{\boxprod}{\mathbin\square}

\newcommand{\pcorner}{\circledast}

\newcommand{\coprodover}[1]{\underset{#1}{\coprod}}

\usepackage{tikz}
\usetikzlibrary{arrows,decorations.pathmorphing}
\usetikzlibrary{backgrounds,positioning}
\usetikzlibrary{fit,petri,shapes.misc}

\tikzset{auto}
\tikzset{empty/.style={circle,inner sep=0pt,minimum size=6mm}}
\tikzset{emptyvt/.style={circle,inner sep=0pt,minimum size=0mm}}

\tikzset{plain/.style={circle,draw,very thick,
inner sep=0pt,minimum size=6mm}}

\tikzset{fatplain/.style={rounded rectangle,draw,very thick,minimum size=6mm}}

\tikzset{bigplain/.style={rounded rectangle,draw,very thick,minimum size=.8cm}}

\tikzset{yellowvt/.style={circle,draw,fill=yellow,very thick,inner sep=0pt,minimum size=6mm}}

\tikzset{bluevt/.style={circle,draw,fill=blue!20,very thick,inner sep=0pt,minimum size=6mm}}

\tikzset{greenvt/.style={circle,draw,fill=green!30,very thick,inner sep=0pt,minimum size=6mm}}

\tikzset{redvt/.style={circle,draw,fill=red!30,very thick,inner sep=0pt,minimum size=6mm}}

\tikzset{arrow/.style={->,thick}}
\tikzset{dashedarrow/.style={->,dashed,thick}}
\tikzset{dottedarrow/.style={->,dotted,thick}}
\tikzset{mapto/.style={|->,thick}}

\tikzset{implies/.style={thick,double,double equal sign distance,-implies}}

\tikzset{line/.style={thick}}
\tikzset{dottedline/.style={dotted,thick}}
\tikzset{dashedline/.style={dashed,thick}}

\tikzset{inputleg/.style={<-,thick}}
\tikzset{outputleg/.style={->,thick}}
\tikzset{dottedinput/.style={<-,dotted,thick}}

\newcommand{\adjoint}{\hspace{-.1cm}
\nicearrow\xymatrix{ \ar@<2pt>[r] & \ar@<2pt>[l]}\hspace{-.1cm}}
\renewcommand{\hookrightarrow}{\nicexy{\ar@{^{(}->}[r] &}}
\newcommand{\nicearrow}{\SelectTips{cm}{10}}
\newcommand{\nicexy}{\nicearrow\xymatrix@C+10pt@R+8pt}
\newcommand{\narrowxy}{\nicearrow\xymatrix@R+10pt}

\renewcommand{\to}{\hspace{-.1cm}\nicearrow\xymatrix@C-.3cm{\ar[r]&}\hspace{-.1cm}}

\newcommand{\tensorunit}{\mathbb{1}}

\newcommand{\Cat}{\mathsf{Cat}}

\renewcommand{\sD}{\mathsf{D}}

\newcommand{\Top}{\mathsf{Top}}

\newcommand{\spancat}{\{\hspace{-.1cm}\nicexy@C-.8cm{-1 & 0 \ar[l] \ar[r] & 1}\hspace{-.1cm}\}}

\newcommand{\proj}{{\scalebox{.5}{$\mathrm{proj}$}}}

\newcommand{\arrowm}{\overrightarrow{\M}}

\newcommand{\arrowmproj}{\arrowm_{\proj}}

\newcommand{\arrowmpp}{\arrowm^{\square}}

\newcommand{\arrowmppproj}{\arrowmpp_{\proj}}

\renewcommand{\lim}{\mathsf{lim}\,}

\DeclareMathOperator{\Ev}{Ev}
\DeclareMathOperator{\Hom}{Hom}

\DeclareMathOperator{\Id}{Id}

\begin{document}

\title{Arrow Categories of Monoidal Model Categories}

\author{David White}
\address{Denison University
\\ Granville, OH}
\email{david.white@denison.edu}

\author{Donald Yau}
\address{The Ohio State University at Newark \\ Newark, OH}
\email{yau.22@osu.edu}

\begin{abstract}
We prove that the arrow category of a monoidal model category, equipped with the pushout product monoidal structure and the projective model structure, is a monoidal model category. This answers a question posed by Mark Hovey, in the course of his work on Smith ideals. As a corollary, we prove that the projective model structure in cubical homotopy theory is a monoidal model structure.
 As illustrations we include numerous examples of non-cofibrantly generated monoidal model categories, including chain complexes, small categories, pro-categories, and topological spaces.
\end{abstract}

\maketitle

%\tableofcontents

%======================
\section{Introduction}
%======================

Based on an unpublished talk by Jeff Smith, in \cite{hovey-smith} Hovey developed a homotopy theory of Smith ideals in monoidal model categories.  Given a symmetric monoidal closed category $\M$, its arrow category Arr($\M)=\arrowm$ becomes a symmetric monoidal closed category when equipped with the pushout product monoidal structure, denoted $\arrowmpp$.  A Smith ideal in $\M$ is  defined as a monoid in $\arrowmpp$.  Equivalently, a Smith ideal is a triple $(R,I,j)$ with $R$ a monoid in $\M$, $I$ an $R$-bimodule, and $j : I \to R$ a map of $R$-bimodules that satisfies an extra compatibility condition.

If $\M$ has a (cofibrantly generated) model structure, then its arrow category inherits a (cofibrantly generated) projective model structure with weak equivalences and fibrations defined entrywise in $\M$.  This is because the arrow category is the category of functors from $I = \{0 \to 1\}$, which is a direct category, to $\M$.  A monoidal model category \cite{ss} is a symmetric monoidal closed category with a model structure that satisfies the pushout product axiom.  In \cite{hovey-smith} (3.1(5)) Hovey showed that if $\M$ is a cofibrantly generated monoidal model category, then its arrow category equipped with the pushout product monoidal structure and the  projective model structure is a monoidal model category.  Furthermore, Hovey expressed the belief that $\arrowmpp$ should be a monoidal model category even without assuming cofibrant generation on $\M$.  Such a result was first achieved in \cite{ps} (Prop. 3.1.8) with an inductive argument that involves decomposing certain maps into composites of pushouts.  The purpose of this short note is to reprove this result with a direct, non-inductive argument, and to give numerous applications of this result.

\begin{main} \label{thm:main}
Suppose $\M$ is a monoidal model category.  Then its arrow category equipped with the pushout product monoidal structure and the projective model structure is a monoidal model category.
\end{main}

This result will be proved below as Theorem \ref{arrowcat-ppax}.  The point is that it is not necessary to assume cofibrant generation on $\M$, as in \cite{hovey-smith} (3.1(5)).   The primary difficulty of verifying the pushout product axiom in $\arrowmpp$, with the projective model structure, is computing the pushout product in $\arrowmpp$ in terms of pushout products in $\M$. The pushout product of two morphisms in $\arrowmpp$ is described in (\ref{alpha-box-beta}).

To deal with these categorical and homotopical difficulties, our proof below involves three rewritings of pushout corner maps.  In the first instance \eqref{teresa}, we rewrite a pushout corner map as the induced map between two pushout squares.  In each of the other two instances \eqref{xi-rewrite} and \eqref{alphabetaf-pc}, a pushout corner map is rewritten as a pushout product.

In Section \ref{sec:not-cg} we will provide a list of examples of monoidal model categories that are not cofibrantly generated, including chain complexes, small categories, spaces, and pro-categories. In Section \ref{sec:applications}, we discuss an application of our main result to cubical homotopy theory, a field that has recently been applied in numerous settings, including Goodwillie calculus, homotopy type theory, classical stable and unstable homotopy theory, rewriting theory, concurrency theory, and the homotopy theory of $C^*$-algebras. A corollary of our main result, Corollary \ref{cor:main}, implies that the natural monoidal product in cubical homotopy theory satisfies the pushout product axiom, and hence provides practitioners of cubical homotopy theory with a more powerful set of tools. 

Cubical homotopy theory begins with a sequence of diagram categories $\M^{I}:=\Arr(\M)$, $\M^{I^{\times 2}}:=\Arr(\M^{I}), \dots, \M^{I^{\times n}}, \dots$, each being the arrow category of the previous step. The objects of $\M^{I^{\times n}}$ are commutative $n$-cubes in $\M$ (i.e., functors from $\{0 \to 1\}^{\times n}$ to $\M$) and maps are commutative $(n+1)$-cubes (i.e., natural transformations between functors from $\{0 \to 1\}^{\times n}$ to $\M$). For each of these diagram categories, the projective model structure defines weak equivalences and fibrations pointwise in $\M$. By iterating Theorem \ref{thm:main}, we obtain the following corollary, discussed further in Section \ref{sec:applications}.

\begin{mainCor} \label{cor:main}
Let $\M$ be any monoidal model category (cofibrantly generated or not). Then all of the diagram categories $\M^{I^{\times n}}$, equipped with their projective model structures, are monoidal model categories.
\end{mainCor}

Another motivation for Theorem \ref{thm:main} is that it is needed in a companion paper \cite{white-yau6}, that uses the theory of \cite{white-yau1} to extend Hovey's work on Smith ideals of ring spectra \cite{hovey-smith} to the context of algebras over general operads, rather than simply monoids. This work builds on \cite{white-commutative}, where commutative Smith ideals were introduced, as commutative monoids in $\arrowmproj$. Using Theorem \ref{thm:main}, we are able to lift operads $\cat O$ to new operads $\arrowo$ that act in $\arrowmproj$. We then work out the homotopy theory of $\arrowo$-algebras, generalizing results of Hovey \cite{hovey-smith}, proving new results regarding left Bousfield localization, and setting up a theory of Smith $\cat O$-ideals in spectra, chain complexes, and the stable module category.

%===============================================
\section{Projective Model Structure on the Arrow Category}
\label{sec:model-arrow}
%======================

In this section we briefly recall some definitions and facts regarding monoidal model categories and arrow categories.  Our main references for model categories are \cite{hirschhorn,hovey,ss}.  In this paper, $(\M, \otimes, \tensorunit, \Hom)$  will be a bicomplete symmetric monoidal closed category with monoidal unit $\tensorunit$ and initial object $\varnothing$.

\begin{definition}\label{ppax}
\begin{enumerate}
\item A model category is \emph{cofibrantly generated} if there are a set $I$ of cofibrations and a set $J$ of trivial cofibrations (i.e., maps that are both cofibrations and weak equivalences) that permit the small object argument (with respect to some cardinal $\kappa$), and a map is a (trivial) fibration if and only if it satisfies the right lifting property with respect to all maps in $J$ (resp., $I$).
\item A symmetric monoidal closed category $\M$ equipped with a model structure is called a \emph{monoidal model category} if it satisfies the following \emph{pushout product axiom} \cite{ss} (3.1): 
\begin{itemize}
\item Given any cofibrations $f:X_0\to X_1$ and $g:Y_0\to Y_1$, the pushout product map
\[\nicexy{(X_0\otimes Y_1) \coprodover{X_0\otimes Y_0} (X_1\otimes Y_0) 
\ar[r]^-{f\boxprod g} & X_1\otimes Y_1}\]
is a cofibration. If, in addition, either $f$ or $g$ is a weak equivalence then $f\boxprod g$ is a trivial cofibration.
\end{itemize}
\end{enumerate}
\end{definition}

We now recall the arrow category and its projective model structure from \cite{hovey-smith}.

\begin{definition}
\begin{enumerate}
\item Given a solid-arrow commutative diagram
\[\nicexy@R-.5cm{A \ar[d] \ar[r] \ar@{}[dr]|-{\mathrm{pushout}} & C \ar[d] \ar@/^1pc/[ddr]|-{g} & \\ 
B \ar[r] \ar@/_1pc/[drr]|-{f} & B \coprodover{A} C 
 \ar@{}[dr]^(.15){}="a"^(.9){}="b" \ar@{.>} "a";"b" |-{f \pcorner g} & \\ && D}\]
in $\M$ in which the square is a pushout, the unique dotted induced map--i.e., the \emph{pushout corner map}--will be denoted by $f \pcorner g$.  The only exception to this notation is when the pushout corner map is actually a pushout product of two maps, in which case we keep the box notation in Def. \ref{ppax}.
\item The \emph{arrow category} $\arrowm$ is the category whose objects are maps in $\M$, in which a map $\alpha : f \to g$ is a commutative square
\begin{equation}\label{map-in-arrowcat}
\nicexy@R-10pt{X_0 \ar[r]^-{\alpha_0} \ar[d]_f & Y_0 \ar[d]^-{g}\\ 
X_1 \ar[r]^-{\alpha_1} & Y_1}
\end{equation}
in $\M$.  We will also write $\Ev_0f = X_0$, $\Ev_1f = X_1$, $\Ev_0 \alpha = \alpha_0$, and $\Ev_1 \alpha = \alpha_1$.
\item The \emph{pushout product monoidal structure} on $\arrowm$ is given by the pushout product
\[\nicexy{(X_0 \otimes Y_1) \coprodover{X_0 \otimes Y_0} (X_1 \otimes Y_0) \ar[r]^-{f \boxprod g} & X_1 \otimes Y_1}\]
for $f : X_0 \to X_1$ and $g : Y_0 \to Y_1$.  The arrow category equipped with this monoidal structure is denoted by $\arrowmpp$.  Its monoidal unit is $\varnothing \to \tensorunit$, and it is a symmetric monoidal closed category.
\item Defining $L_0(X) = (\Id : X \to X)$ and $L_1(X) = (\varnothing \to X)$ for $X \in \M$, there are adjunctions
\begin{equation}\label{lev}
\nicexy{\M \ar@<2pt>[r]^-{L_0} & \arrowm \ar@<2pt>[l]^-{\Ev_0} & \M \ar@<2pt>[r]^-{L_1} & \arrowm \ar@<2pt>[l]^-{\Ev_1}}
\end{equation}
with left adjoints on top.
\end{enumerate}
\end{definition}

Of course, the arrow category is also the category of functors from the category $\{0 \to 1\}$, with two objects and one non-identity arrow, to $\M$.  The following result about the projective model structure is from \cite{hovey-smith} (3.1).

\begin{theorem}\label{hovey-projective}
Suppose $\M$ is a model category.
\begin{enumerate}
\item There is a model structure on $\arrowm$, called the \emph{projective model structure}, in which a map $\alpha : f \to g$ as in \eqref{map-in-arrowcat} is a weak equivalence (resp., fibration) if and only if $\alpha_0$ and $\alpha_1$ are weak equivalences (resp., fibrations) in $\M$.  A map $\alpha$ is a (trivial) cofibration if and only if $\alpha_0$ and the pushout corner map
\[\nicexy{X_1 \coprodover{X_0} Y_0 \ar[r]^-{\alpha_1 \pcorner g} & Y_1}\]
are (trivial) cofibrations in $\M$.  The arrow category equipped with the projective model structure is denoted by $\arrowmproj$.
\item If $\M$ is a cofibrantly generated monoidal model category, then $\arrowmpp$ equipped with the projective model structure is a monoidal model category.
\end{enumerate}
\end{theorem}

\begin{remark}
In (1) above the projective model structure on $\arrowm$ is the special case of \cite{hovey} (5.1.3) for the direct category $\{0 \to 1\}$. If $\alpha$ is a (trivial) cofibration, then $\alpha_1$ is also a (trivial) cofibration.  If $\M$ is cofibrantly generated with generating cofibrations $I$ and generating trivial cofibrations $J$, then $\arrowmproj$ is cofibrantly generated with generating cofibrations $L_0I \cup L_1I$ and generating trivial cofibrations $L_0J \cup L_1J$ \cite{hovey} (5.1.8).
\end{remark}

%===================================================
\section{Arrow Category as a Monoidal Model Category}

In \cite{hovey-smith} (immediately after 3.1) Hovey stated that the last statement in Theorem \ref{hovey-projective} should be true even without assuming $\M$ is cofibrantly generated.  In this section, we prove that this is indeed the case. 

\begin{theorem}\label{arrowcat-ppax}
Suppose $\M$ is a monoidal model category. Then $\arrowmpp$ equipped with the projective model structure is a monoidal model category.
\end{theorem}

\begin{proof}
We already know that $\arrowmpp$ is a symmetric monoidal closed category equipped with the projective model structure.  We must show that it satisfies the pushout product axiom.  Suppose $\alpha : f_V \to f_W$ and $\beta : f_X \to f_Y$,
\begin{equation}\label{alpha-and-beta}
\nicexy@R-.3cm{V_0 \ar[d]_-{f_V} \ar[r]^-{\alpha_0} & W_0 \ar[d]^-{f_W} & X_0 \ar[d]_-{f_X} \ar[r]^-{\beta_0} & Y_0 \ar[d]^-{f_Y}\\
V_1 \ar[r]^-{\alpha_1} & W_1 & X_1 \ar[r]^-{\beta_1} & Y_1}
\end{equation}
are two maps in $\arrowmpp$.  Their pushout product in $\arrowmpp$ is the map
\[\nicexy{(f_W \boxprod f_X) \coprodover{f_V \boxprod f_X} (f_V \boxprod f_Y) \ar[r]^-{\alpha \boxprod_2 \beta} & f_W \boxprod f_Y}\]
in which $\boxprod$ (resp., $\boxprod_2$) is the pushout product in $\M$ (resp., $\arrowmpp$).  To simplify the notation, in the diagrams below we will omit writing $\otimes$, so $V_0 X_0$ means $V_0 \otimes X_0$, etc. Unraveling the various pushout products, $\alpha \boxprod_2 \beta$ is the commutative square
\begin{equation}\label{alpha-box-beta}
\begin{small}
\nicexy{\Bigl(W_1 X_0 \coprodover{W_0 X_0} W_0 X_1\Bigr) \coprodover{\bigl(V_1X_0 \coprodover{V_0X_0} V_0X_1\bigr)} \Bigl(V_1Y_0 \coprodover{V_0Y_0} V_0Y_1\Bigr) \ar[d]_-{(f_W \boxprod f_X) \coprodover{f_V \boxprod f_X} (f_V \boxprod f_Y)} \ar[r]^-{\zeta} & W_1Y_0 \coprodover{W_0Y_0} W_0Y_1 \ar[d]^-{f_W \boxprod f_Y}\\
W_1X_1 \coprodover{V_1X_1} V_1Y_1 \ar[r]^-{\alpha_1 \boxprod \beta_1} & W_1Y_1}
\end{small}
\end{equation}
in $\M$.  

To prove the pushout product axiom in $\arrowmpp$ equipped with the projective model structure, suppose $\alpha$ is a cofibration and $\beta$ is a (trivial) cofibration in $\arrowmproj$.  The case with $\alpha$ a trivial cofibration and $\beta$ a cofibration is proved by essentially the same argument.  To show that $\alpha \boxprod_2 \beta$ is a (trivial) cofibration in $\arrowmproj$, we must show that:
\begin{enumerate}
\item The top horizontal map $\zeta$ in \eqref{alpha-box-beta} is a (trivial) cofibration in $\M$.
\item The pushout corner map 
\begin{equation}\label{alphabeta-pc}
\nicexy@C+1.5cm{\Bigl(W_1X_1 \coprodover{V_1X_1} V_1Y_1\Bigr) \coprodover{Z} \Bigl(W_1Y_0 \coprodover{W_0Y_0} W_0Y_1\Bigr) \ar[r]^-{(\alpha_1 \boxprod \beta_1) \pcorner (f_W \boxprod f_Y)} & W_1Y_1}
\end{equation}
of the commutative diagram \eqref{alpha-box-beta} is a (trivial) cofibration in $\M$, where $Z$ denotes the object in the upper left corner in \eqref{alpha-box-beta}.
\end{enumerate}
We will prove statement (1) in Lemma \ref{zeta-cof} and statement (2) in Lemma \ref{alphabeta-pc-cof} below.
\end{proof}

\begin{lemma}\label{zeta-cof}
The top horizontal map $\zeta$ in \eqref{alpha-box-beta} is a (trivial) cofibration in $\M$.
\end{lemma}

\begin{proof}
First note that $\zeta$ is a pushout corner map.  By the commutation of colimits, we may rewrite $\zeta$ as the unique induced map in the commutative cube
\begin{equation}\label{teresa}
\begin{scriptsize}
\nicexy@R-.4cm@C-1cm{
W_0X_0 \coprodover{V_0X_0} V_0Y_0 \ar[ddd] \ar@{}[ddr]^(0.2){}="a"^(.9){}="b" \ar "a";"b" |-{\alpha_0 \boxprod \beta_0} \ar@{}[drr]|-{\mathrm{pushout}} \ar[rrr] &&& W_0X_1 \coprodover{V_0X_1} V_0Y_1 \ar@{.>}[ddd] \ar[ddr]|-{\alpha_0 \boxprod \beta_1} \ar[dl] &\\
 && P_{\mathrm{Top}}\ar@{..>}[ddd] \ar[drr]^(0.7){\xi} && \\
& W_0Y_0 \ar[ddd] \ar[rrr] \ar[ur] &&& W_0Y_1 \ar[ddd]\\
W_1X_0 \coprodover{V_1X_0} V_1Y_0 \ar@{}[ddr]^(0.2){}="a"^(.9){}="b" \ar "a";"b"|-{\alpha_1 \boxprod \beta_0} \ar'[r][rrr]  &\ar@{}[dr]|-{\mathrm{pushout}} && 
\Bigl(W_1X_0 \coprodover{V_1X_0} V_1Y_0\Bigr) \coprodover{\bigl(W_0X_0 \coprodover{V_0X_0} V_0Y_0\bigr)} \Bigl(W_0X_1 \coprodover{V_0X_1} V_0Y_1\Bigr) \ar@{}[ddr]^(0.24){}="a"^(.9){}="b" \ar "a";"b" |-{\zeta} \ar[dl]^(.7){\delta_0} &\\
 && P_{\mathrm{Bot}} \ar@{..>}[drr]^{\delta_1}&& \\
& W_1Y_0 \ar[rrr] \ar[ur] &&& W_1Y_0 \coprodover{W_0Y_0} W_0Y_1}
\end{scriptsize}
\end{equation}
in $\M$ with both the back and the front faces pushouts, and with $P_{Top}$ (resp. $P_{Bot}$) as the pushout of the displayed spans in the top (resp. bottom) faces, as shown above. Furthermore, the diagonal face, featuring $P_{Top}, P_{Bot}, \xi,$ and $\delta_1$ is a pushout square. By colimit commutation, the relevant map, $\zeta$, is the composition $\delta_1 \circ \delta_0$. Since $\delta_0$ is a pushout of $\alpha_1 \boxprod \beta_0$, it is a (trivial) cofibration in $\M$. Since $\delta_1$ is a pushout of $\xi$, it suffices to prove that $\xi$ is a (trivial) cofibration in $\M$. We can rewrite $\xi$ as the pushout product $\alpha_0 \boxprod (\beta_1 \pcorner f_Y)$ in the diagram
\begin{equation}\label{xi-rewrite}
\nicexy@R-.5cm@C-.4cm{V_0\Bigl(X_1 \coprodover{X_0} Y_0\Bigr) \ar@{}[dr]|-{\mathrm{pushout}} \ar[d]_-{(\alpha_0,\Id)}  \ar[r]^-{(\Id, \beta_1 \pcorner f_Y)} & V_0Y_1 \ar[d] \ar@/^2pc/[ddr]^-{(\alpha_0, \Id)} &\\
W_0\Bigl(X_1 \coprodover{X_0} Y_0\Bigr) \ar[r] \ar@/_1pc/[drr]_-{(\Id, \beta_1 \pcorner f_Y)} & \Bigl[W_0\bigl(X_1 \coprodover{X_0} Y_0\bigr)\Bigr] \coprodover{\bigl[V_0(X_1 \coprodover{X_0} Y_0)\bigr]} (V_0Y_1) \ar@{}[dr]^(0.37){}="a"^(.9){}="b" \ar "a";"b" |-{\xi}  &\\ && W_0Y_1}
\end{equation}
in $\M$.  Indeed, the pushout in the previous diagram has the same universal property as the pushout of the top face of the cube \eqref{teresa}, and the pushout corner map to $W_0Y_1$ corresponds to the previous pushout product.  Since the map $\alpha_0$ is a cofibration and since the pushout corner map $\beta_1 \pcorner f_Y$ is a (trivial) cofibration in $\M$, their pushout product $\xi$ is a (trivial) cofibration in $\M$ by the pushout product axiom.
\end{proof}

\begin{remark}
An alternative way to prove Lemma \ref{zeta-cof} is to consider the Reedy category $\sD = \spancat$ with three objects, a map $0 \to -1$ that lowers the degree, a map $0 \to 1$ that raises the degree, and no other non-identity maps.  There is a Reedy model structure \cite{hovey} (5.2.5) on the diagram category $\M^{\sD}$ in which weak equivalences are defined entrywise.  A map $h : A \to B$ in $\M^{\sD}$,
\begin{equation}\label{hAB}
\nicexy@R-.3cm{A_{-1} \ar[d]_-{h_{-1}} & A_0 \ar[l]\ar[d]^-{h_0} \ar[r] & A_1 \ar[d]^-{h_1}\\
B_{-1} & B_0 \ar[l] \ar[r]|-{g} & B_1}
\end{equation}
is a Reedy (trivial) cofibration if and only if the maps $h_{-1}$, $h_0$, and the pushout corner map $g \pcorner h_1 : B_0 \coprod_{A_0} A_1 \to B_1$ are (trivial) cofibrations in $\M$.  Furthermore, there is a Quillen adjunction
\begin{equation}\label{reedy-adjunction}
\nicexy@C+.5cm{\M^{\sD} \ar@<2pt>[r]^-{\colim} & \M \ar@<2pt>[l]^-{\mathrm{constant}}}
\end{equation}
in which the left Quillen functor $\colim$ is the pushout \cite{hovey} (proof of 5.2.6).

To show that the induced map $\zeta$ in \eqref{teresa} is a (trivial) cofibration in $\M$, one can use the Quillen adjunction \eqref{reedy-adjunction}.  It is enough to show that the diagram consisting of the left and the top faces of the cube \eqref{teresa}--which has the form \eqref{hAB}--is a Reedy (trivial) cofibration in $\M^{\sD}$.  Since the maps $\alpha_0$ and $\alpha_1$ are cofibrations and since $\beta_0$ is a (trivial) cofibration in $\M$, the pushout products $\alpha_1 \boxprod \beta_0$ and $\alpha_0 \boxprod \beta_0$ are (trivial) cofibrations in $\M$ by the pushout product axiom. We thank the referee for suggesting the simplified proof of Lemma \ref{zeta-cof} given above, to avoid the need for Reedy categories.
\end{remark}

\begin{lemma}\label{alphabeta-pc-cof}
The pushout corner map $(\alpha_1 \boxprod \beta_1) \pcorner (f_W \boxprod f_Y)$ in \eqref{alphabeta-pc} is a (trivial) cofibration in $\M$.
\end{lemma}

\begin{proof}
There is a commutative square
\begin{equation}\label{alphabetaf-pc}
\nicexy@R-.6cm@C+1.5cm{\Bigl(W_1X_1 \coprodover{V_1X_1} V_1Y_1\Bigr) \coprodover{Z} \Bigl(W_1Y_0 \coprodover{W_0Y_0} W_0Y_1\Bigr) \ar[r]^-{(\alpha_1 \boxprod \beta_1) \pcorner (f_W \boxprod f_Y)} \ar[d]_-{\cong} & W_1Y_1 \ar[d]^-{=}\\
W_1\Bigl(X_1\coprodover{X_0} Y_0\Bigr) \coprodover{\bigl(V_1 \coprodover{V_0} W_0\bigr)\bigl(X_1 \coprodover{X_0} Y_0\bigr)} \Bigl(V_1 \coprodover{V_0} W_0\Bigr)Y_1 \ar[r]^-{(\alpha_1 \pcorner f_W) \boxprod (\beta_1 \pcorner f_Y)} & W_1Y_1}
\end{equation}
with $\alpha_1 \pcorner f_W$ (resp., $\beta_1 \pcorner f_Y$) the pushout corner map of $\alpha$ (resp., $\beta$) in \eqref{alpha-and-beta} and the bottom horizontal map the pushout product of $\alpha_1 \pcorner f_W$ and $\beta_1 \pcorner f_Y$.  The vertical isomorphism on the left comes from the fact that the pushout in the lower left corner has the same universal property as the pushout in the upper left corner.  Since $\alpha_1 \pcorner f_W$ is a cofibration and since $\beta_1 \pcorner f_Y$ is a (trivial) cofibration in $\M$, their pushout product--the bottom horizontal map in \eqref{alphabetaf-pc}--is a (trivial) cofibration in $\M$ by the pushout product axiom.
\end{proof}

%==========================================
\section{Examples : Non-Cofibrantly Generated Monoidal Model Categories}
\label{sec:not-cg}

In this section we consider examples of monoidal model categories that are \emph{not} cofibrantly generated.  By Theorem \ref{arrowcat-ppax} each such category yields a monoidal model structure in its arrow category with the pushout product monoidal structure and the projective model structure. We note that these examples are not pathological; all arose naturally in homotopy theoretic investigations, and none are known to have Quillen equivalent cofibrantly generated model structures. Indeed, for several of these examples, there cannot be {\em any} cofibrantly generated model structure encoding its homotopy theory, as we prove in Section \ref{subsec:no-model}.

\subsection{Christensen-Hovey Model Structure on Integral Chain Complexes} \label{subsec:absolute}
The category $\Chz$ of chain complexes of abelian groups admits the \emph{absolute model structure} \cite{christensen-hovey}.  The weak equivalences in $\Chz$ are the chain homotopy equivalences.  Cofibrations (resp., fibrations) are the degreewise split monomorphisms (resp., degreewise split epimorphisms).  Equipped with the absolute model structure, $\Chz$ is a \emph{non}-cofibrantly generated monoidal model category \cite{christensen-hovey} (Example 3.4 and Cor. 5.12).

\subsection{Barthel-May-Riehl Model Structure on DG-Modules} \label{subsec:dga}
For a commutative ring $R$, the category of differential graded $R$-modules admits the \emph{$r$-model structure} $dgRmod_r$ \cite{six-model} (1.14 and 1.15).  Analogous to the previous example, its cofibrations (resp., fibrations) are the degreewise split monomorphisms (resp.,  degreewise split epimorphisms).  Its weak equivalences are the chain homotopy equivalences.  This is a \emph{non}-cofibrantly generated monoidal model category.

\subsection{Ad\'{a}mek-Herrlich-Rosick\'{y}-Tholen Model Structure on Small Categories} \label{subsec:AHRT-cat}
The category $\Cat$ of all small categories has a \emph{non}-cofibrantly generated model structure \cite{ahrt} (2.3, 3.5, and 3.7) in which every map is a weak equivalence and cofibrations are the full functors.  Trivial fibrations are the topological functors.  We will call this the \emph{AHRT model structure} on $\Cat$.  With Cartesian product as the monoidal product, the AHRT model structure on $\Cat$ is a monoidal model category because the pushout product of two full functors is again a full functor. 
Similarly, the AHRT model structure on the category of small posets is not cofibrantly generated \cite{ahrt} (3.4).

\subsection{Str{\o}m Model Structure on Compactly Generated Spaces}
The category $\Top$ of compactly generated spaces has a Str{\o}m model structure \cite{strom} with homotopy equivalences as weak equivalences, closed Hurewicz cofibrations as cofibrations, and Hurewicz fibrations as fibrations.  This is a monoidal model category \cite{may99} (Chapter 6.4) and is \emph{not} cofibrantly generated \cite{raptis} (Remark 4.7). 

\subsection{Pro-categories} \label{subsec:pro}

For any category $\cat C$, the pro-category pro-$\cat C$ has objects cofiltered diagrams $X = \{x_a\}$ of objects of $\cat C$, and morphisms 

\[
\text{Hom}_{\text{pro}-{\cat C}}(X, Y) = \lim_{\beta} \colim_{\alpha} \cat C(x_\alpha, y_\beta)
\]

Isaksen \cite{isaksen} built the strict model structure on pro$-\cat C$ whenever $\cat C$ is a proper model category. However, pro-$\cat C$ is almost never cofibrantly generated (rather, it is often fibrantly generated). In particular, it is not cofibrantly generated when $\cat C$ is $sSet$ \cite{isaksen} (Section 5).

If $\cat C$ is a tensor model category, i.e. a monoidal model category such that functors $C\otimes -$ and $-\otimes C$ preserve weak equivalences for all cofibrant $C$, then pro-$\cat C$ is also a tensor model category, with the levelwise tensor product \cite{fausk-isaksen} (Proposition 12.7), although pro-$\cat C$ is almost never a closed category. Hence, for most proper, tensor model categories $\cat C$, the category pro$-\cat C$ is a monoidal model category that is not cofibrantly generated.

\subsection{On presentable $\infty$-categories} \label{subsec:no-model}

Recall that a model category $\M$ is called {\em combinatorial} if $\M$ is locally presentable as a category and cofibrantly generated as a model category. Proposition A.3.7.6 of \cite{htt} demonstrates that the $\infty$-category associated to a combinatorial model category is presentable. Furthermore, in a combinatorial model category, the classes of weak equivalences and trivial fibrations are closed under sufficiently large filtered colimits, by \cite{dugger} (Proposition 2.3). The category of small categories, of chain complexes over a ring, and of differential graded $R$-modules are all locally presentable. So the only obstacle to the model categories of Sections \ref{subsec:absolute}, \ref{subsec:dga}, and \ref{subsec:AHRT-cat}, being combinatorial is cofibrant generation. We now argue that there can be no combinatorial model structure for these three homotopy theories, as well as the homotopy theory of Section \ref{subsec:pro}.

\begin{proposition}
The homotopy theory encoded by the absolute model structure of \cite{christensen-hovey} is a non-presentable $\infty$-category and cannot admit any combinatorial model.
\end{proposition}

\begin{proof}
Corollary 1.4.4.2 in \cite{lurie-higher-algebra} implies that the homotopy category of a presentable stable $\infty$-category must be well-generated as a triangulated category. However, the homotopy category of $\Chz$ is $K(\mathbb{Z})$, and is known not to be well-generated as discussed in \cite{christensen-hovey} (5.4). If there was any combinatorial model for this homotopy theory, it would imply $K(\mathbb{Z})$ is well-generated, a contradiction. 
\end{proof}

The same argument implies that the $r$-model structures of \ref{subsec:dga} cannot, in general, admit combinatorial models.

\begin{proposition}
The homotopy theory encoded by the AHRT model structure on $Cat$ of \cite{ahrt} is a non-presentable $\infty$-category and cannot admit any combinatorial model.
\end{proposition}

\begin{proof}
Consider the class of trivially fibrant objects in the category $Pos$ of posets, with the AHRT model structure. Proposition 3.4 of \cite{ahrt} implies that this class is not closed under $\lambda$-filtered colimits, no matter how large $\lambda$ is allowed to become. On the level of the $\infty$-category associated to $Pos$, this implies the class of equivalences is not accessible. It follows that the $\infty$-category cannot be presentable, hence cannot admit any combinatorial model. The argument of \cite{ahrt} (Proposition 3.5) implies the same conclusion for $Cat$ with the AHRT model structure.
\end{proof}

Pro-categories pro$-\cat C$ are not presentable, and often not copresentable either \cite{htt} (Chapter 7), hence the homotopy theory of Section \ref{subsec:pro} is not presentable either.

\begin{remark}
There are several other examples of non-cofibrantly generated model structures, that we did not include because we did not know whether or not they were monoidal model categories. For example, all of the following model structures are not cofibrantly generated:
\begin{enumerate}
\item  the trivial model structure of \cite{lack} (Proposition 4.18) on the 2-category of arrows in $Cat$,
\item  the projective model structure of \cite{bcr} (Theorem 2.4) on small functors from a simplicial category $\cat K$ to simplicial sets, 
\item the model structure of \cite{chorny} (Theorem 1.2) on the arrow category of simplicial sets, 
\item the localization of the category of small functors from $sSet$ to itself constructed in \cite{chorny-rosicky} (Example 3.16),
\item the weak factorization system on $Cat$ of \cite{bourke-garner} (Section 6.2)

\end{enumerate}

As categories of functors, the first four examples above can be endowed with the Day convolution product. However, the authors do not know if this product turns these model structures into monoidal model categories.
\end{remark}

\section{Application to Cubical Homotopy Theory} \label{sec:applications}

\subsection{Cubical Diagram Categories of Monoidal Model Categories}\label{iterate}

Cubical homotopy theory is an alternative to simplicial homotopy theory that has recently found powerful applications in Goodwillie calculus \cite{munson-volic}, Blakers-Massey theorems \cite{ching-harper}, homotopy type theory \cite{awodey, cisinski, cohen}, rewriting theory \cite{lucas}, concurrency theory \cite{meadows}, the homotopy theory of $C^*$-algebras \cite{ostvaer}, and classical homotopy theory \cite{brown}. Jardine \cite{jardine-cubical} produced the first model structure on cubical sets, and pointed out several advantages of the cubical setting over simplicial sets. 

The idea of cubical sets is to replace the simplex category $\Delta$ by the cubical category $\boxprod$. A cubical set is a functor $X:\boxprod^{op}\to Set$, i.e. a collection of sets $(X_n)_{n\in \mathbb{N}}$, where $X_n$ is thought of as the set of $n$-cubes. Jardine \cite{jardine-cubical} proved that the homotopy theory of cubical sets agrees with that of topological spaces. Just as one can consider simplicial objects in a model category $\M$, so can one consider cubical objects in $\M$, namely functors $X:\boxprod^{op}\to \M$, where $X_n$ encodes $n$-cubes in $\M$.

As a category of functors, the category of cubical objects in $\M$ admits the Day convolution product. The monoidal structure on cubical objects has had powerful applications in numerous settings, notably in \cite{brown}, \cite{isaacson}, and \cite{ostvaer}. This product agrees with the levelwise product, where the product in level $n$ is the pushout product, obtained by viewing the category of $n$-cubes in $\M$ as the arrow category of the category of $n-1$ cubes. We denote by $\M^{I^{\times n}}$ the category of $n$-cubes; its objects are commutative $n$-cubes in $\M$ and its morphisms are commutative $n+1$ cubes. For example, $\M^{I^{\times 2}}$ is the arrow category of $\arrowmppproj$. Its objects are commutative squares in $\M$ and its morphisms are commutative cubes in $\M$. An example of such a morphism  $\gamma : (f_V \to f_W) \to (f_X \to f_Y)$ is displayed below.
\[\begin{small}
\nicexy@R-.4cm{V_0 \ar[dd]_-{f_V} \ar[rr]^-{\alpha_0} \ar[dr]|-{\gamma_{00}} && W_0 \ar[dr]^-{\gamma_{10}} \ar'[d][dd]|(.55){f_W} &\\
& X_0 \ar[dd]|(.25){f_X} \ar[rr]|(.25){\beta_0} && Y_0 \ar[dd]^-{f_Y}\\
V_1 \ar[dr]_-{\gamma_{01}} \ar'[r][rr]|(.5){\alpha_1} && W_1 \ar[dr]|-{\gamma_{11}} &\\
& X_1 \ar[rr]|-{\beta_1} && Y_1}
\end{small}\]
The map $\gamma$ is a weak equivalence (resp., fibration) in $\M^{I^{\times 2}}$ with the projective model structure if and only if each of the four maps $\gamma_{ij}$ is a weak equivalence (resp., fibration) in $\M$. The projective model structure on $\M^{I^{\times n}}$, for $n > 2$, is defined similarly.

An inductive argument, using Theorem \ref{thm:main} for the base $n=1$, verifies the pushout product axiom on $\M^{I^{\times n}} = \Arr(\M^{I^{\times (n-1)}})$ by appealing to Theorem \ref{thm:main} applied to $\M^{I^{\times (n-1)}}$. This proves Corollary \ref{cor:main}.

Despite the failure of the model structures in Section \ref{sec:not-cg} to be cofibrantly generated, Corollary \ref{cor:main} allows for a monoidal cubical homotopy theory to be built in these settings. We conclude with a corollary, summarizing the considerations of the previous two sections.

\begin{corollary}
Let $\M$ be any of the examples in Section \ref{sec:not-cg}, i.e. $\Chz$, $dgRmod_r$, $Cat$, $Top$, or pro-$\cat C$ (for a tensor model category $\cat C$). Then, by Corollary \ref{cor:main}, the projective model structure on each $\M^{I^{\times n}}$ satisfies the pushout product axiom for every $n$, and hence so does $\M^{\boxprod^{op}}$.
\end{corollary}

\textbf{Acknowledgments:} The authors would like to thank the referee for several helpful suggestions, and would like to thank Dmitri Pavlov and Jakob Scholbach for pointing out their result in \cite{ps}.

\end{document}